\documentclass[a4paper, 12pt, oneside]{article}
\usepackage{graphicx}
\usepackage{curves}
\usepackage{amsmath}
\usepackage{amssymb}
\usepackage{amsfonts}
\usepackage{amsthm}
\usepackage{longtable}
\usepackage{multirow}
\usepackage{color}

\usepackage{enumitem} 
\usepackage{hyperref} 

\newtheorem{lemma}{Lemma}[section]
\newtheorem{theorem}[lemma]{Theorem}
\newtheorem{proposition}[lemma]{Proposition}

\newcommand{\Aut}{{\rm Aut}}

\newcommand{\shortm}{\!-\!}

\begin{document}

\title{Answers to questions about medial layer graphs of self-dual regular and chiral polytopes}

\author{
Marston Conder
\\
{\normalsize Department of Mathematics, University of Auckland,}\\[-4pt]
{\normalsize Private Bag 92019, Auckland 1142, New Zealand} \\[-4pt]
{\normalsize Email: m.conder@auckland.ac.nz}
\\[+12pt]
Isabelle Steinmann
\\
{\normalsize Eckhart Hall, University of Chicago, 5734 S University Ave, } \\[-4pt]
{\normalsize Chicago, IL 60637, USA} \\[-4pt]
{\normalsize Email: isteinmann@uchicago.edu}
}

\date{}

\maketitle

\begin{abstract}
{An abstract $n$-polytope $\mathcal{P}$ is a partially-ordered set which captures important properties of a geometric polytope, for any dimension $n$. For even $n \ge 2$, the incidences between elements in the middle two layers of the Hasse diagram of $\mathcal{P}$ give rise to the medial layer graph of  $\mathcal{P}$, denoted by $\mathcal{G} = \mathcal{G}(\mathcal{P})$. If $n=4$, and $\mathcal{P}$ is both highly symmetric and self-dual of type $\{p,q,p\}$, then a Cayley graph $\mathcal{C}$ covering $\mathcal{G}$ can be constructed on a group of polarities of $\mathcal{P}$. 
In this paper we address some open questions about the relationship between $\mathcal{G}$ and $\mathcal{C}$ that were raised in a 2008 paper by Monson and Weiss, and describe some interesting examples of these graphs. In particular, we give the first known examples of improperly self-dual chiral polytopes of type $\{3,q,3\}$, which are also 
among the very few known examples of highly symmetric self-dual finite polytopes that do not admit a polarity. 
Also we show that if $p=3$ then $\mathcal{C}$ cannot have a higher degree of $s$-arc-transitivity than $\mathcal{G}$, and we present a family of regular $4$-polytopes 
of type $\{6,q,6\}$ for which the vertex-stabilisers in the automorphism group of $\mathcal{C}$ are larger than those for $\mathcal{G}$.} 
\end{abstract}


\section{Introduction}
\label{sec:intro} 

Abstract polytopes are combinatorial structures which mimic the general properties of classical geometric polytopes, such as the Platonic solids,
and yet need not be convex or have natural geometric realisations.   
An abstract polytope is simply a partially-ordered set, the elements of which are assigned a rank representing their dimension, 
and with incidence satisfying certain conditions that arise naturally in a geometric setting. 
    
A maximal chain in this poset is called a {\em flag}, and from the defining properties of an abstract polytope (given in Section~\ref{sec:polytopes}),
every automorphism of such a poset is uniquely determined by its effect on any given flag.  Some $n$-polytopes, called equivelar polytopes, can be assigned a type $\{p_1,\ldots,p_n\}$ that encodes the structure of the 2-sections of the polytope, which are isomorphic when they involve elements of the same ranks. 
    
Two important classes of equivelar polytopes consist of the regular polytopes, which possess the highest degree of symmetry (in that the automorphism group is transitive on the flags), and the chiral polytopes, which have maximal `rotational' symmetry but no reflectional symmetry (and so are not quite regular).  Many `classical' polytopes and other early examples of polytopes are regular, so regular polytopes have been studied extensively. A comprehensive description of regular polytopes is given in McMullen and Schulte's book \textit{Abstract Regular Polytopes}~\cite{McMullenSchulte}. On the other hand, chiral $n$-polytopes do not exist for $n = 2$, and are challenging to find for all except small values of $n \ge 3$, and cannot be convex. For these reasons, they are less well understood. Nevertheless much information about chiral polytopes can be found in a paper by Schulte and Weiss \cite{SchulteWeiss}. 
    
    
Following work by Monson and Weiss in \cite{MonsonWeiss-2008}, here we will investigate two graphs associated with finite and self-dual regular or chiral $4$-polytopes of type $\{p,q,p\}$.  Our investigation will involve some aspects of the study of finite  symmetric (arc-transitive) graphs, which we will describe in Section \ref{subsec:cubic graphs}. 
    
The first graph of interest here is the medial layer graph $\mathcal{G}$, which is constructed from the middle two layers of the Hasse diagram of such a polytope $\mathcal{P}$. This graph $\mathcal{G}$ is bipartite and $p$-valent, and by self-duality of $\mathcal{P}$, it admits automorphisms that swap its two parts. Medial layer graphs were first studied (on the abstract level) in \cite{MonsonWeiss-2007}, where it was shown that when $p=3$ the medial layer graph $\mathcal{G}$ is 3-arc-regular if $\mathcal{P}$ is regular, and 2-arc-regular if $\mathcal{P}$ is chiral.  As a partial converse,  it was also shown how to construct a group $G$ (called $\Gamma$ in \cite{MonsonWeiss-2007}) from a finite connected 2- or 3-arc-regular cubic graph, such that if the generators of $G$ satisfy certain conditions, then $G$ is isomorphic to the automorphism group of a finite self-dual regular or chiral polytope of type $\{3,q,3\}$.
    
The second graph of interest 
is a Cayley graph $\mathcal{C}$ that covers the medial layer graph, as constructed in \cite{MonsonWeiss-2008}. This construction is possible because  (properly) self-dual regular or chiral $4$-polytopes of type $\{p,q,p\}$ possess a polarity (a duality of order $2$) from which $p$ polarities can be created to form the Cayley group of $\mathcal{C}$. Monson and Weiss described in \cite{MonsonWeiss-2008} various connections between $\mathcal{G}$ and $\mathcal{C}$ in detail, and also gave several interesting examples of these graphs for $p=3$, $4$ and $5$. 
    
 In this paper we address some of the open questions raised by Monson and Weiss near the end of their 2008 paper \cite{MonsonWeiss-2008}.  
 In particular, in Section \ref{sec:ISD chiral 3q3} we exhibit the first known examples of chiral polytopes of type $\{3,q,3\}$ that do not possess a polarity that preserves the orbits of flags under the natural action of the automorphism group of the polytope.  (These are also two of the very few known examples of highly symmetric self-dual finite polytopes that do not admit any polarity.  No such polytopes were known at the time their existence was effectively questioned by Schulte and Weiss in their 1991 paper \cite{SchulteWeiss}.)  
In Section \ref{sec: arc transitivity of G and C} we prove that if $p=3$, and $\mathcal{G}$ is $t$-arc-transitive and $\mathcal{C}$ is $s$-arc-transitive, then $s\leq t$, 
and in Section \ref{sec: larger stabiliser} we present an infinite family of directly-regular polytopes of type $\{6,q,6\}$ for which the vertex-stabilisers in the automorphism group of the Cayley graph $\mathcal{C}$ are larger than those for the medial layer graph $\mathcal{G}$. 

Before all that, in Sections \ref{sec:ATgraphs},  \ref{sec:polytopes} and \ref{sec:mediallayergraphs},  we give relevant background on arc-transitive graphs, abstract polytopes 
(including regular and chiral polytopes) and their automorphism groups, and on medial layer graphs, 
and at the end we make some final remarks in Section~\ref{sec:final remarks}.  
More details are available in the second author's University of Auckland Masters thesis~\cite{Steinmann-thesis}.

\section{Arc-transitive graphs}
\label{sec:ATgraphs}

\subsection{Preliminary definitions}
\label{subsec:preliminaries}

All graphs in this paper are assumed to be finite, undirected and simple (with no loops or multiple edges). 

An \textit{arc} in a graph is an ordered edge, or equivalently, an ordered pair of adjacent vertices, 
and a graph is called \textit{arc-transitive} (or \textit{symmetric}) if its automorphism group is transitive on the set of all arcs.
More generally, an $s$\textit{-arc} in a graph is a sequence $(v_0,\ldots,v_s)$ of vertices  such that every two consecutive $v_i$ are adjacent, and every three consecutive $v_i$ are distinct. (In other words, an $s$-arc is a walk of length $s$ in which an edge is never immediately followed by the reverse of that edge.)  

A group $G$ of automorphisms of a graph $\Gamma$ is said to act $s$-arc-transitively on $\Gamma$ if it acts transitively on the set of all $s$-arcs in $\Gamma$, 
and $s$-\textit{arc-regularly} on $\Gamma$ if $G$ acts regularly (that is, sharply-transitively) on the $s$-arcs of $\Gamma$. 
The graph $\Gamma$ is then called $s$\textit{-arc-transitive} (respectively $s$-\textit{arc-regular}) if $\Aut(\Gamma)$ acts $s$-arc-transitively 
(resp. $s$-arc-regularly) on $\Gamma$. 
In particular, 0-arc-transitive means vertex-transitive, and 1-arc-transitive means arc-transitive (or symmetric). 

Note that while arc-transitivity implies edge-transitivity, the converse does not hold. 
Also an $s$-arc-transitive graph might not be $(s-1)$-arc-transitive:  the path graph on three vertices is edge-transitive and 2-arc-transitive, 
yet is neither vertex- nor arc-transitive. 
But an $s$-arc-transitive graph that is not a tree is $t$-arc-transitive whenever $0 \leq t \leq s$, and hence is vertex-transitive and therefore regular. 
Similarly, $s$-arc-regularity implies $s$-arc-transitivity, but the converse is not true:  
a graph might be $s$-arc-transitive for some $s$ but not be $t$-arc-regular for any $t$,  
as shown by the $2$-arc-transitive complete graph $K_5$.

    
    
\subsection{Symmetric cubic graphs}
\label{subsec:cubic graphs}

A focus of much of this paper will be on symmetric 3-valent graphs.  
These are particularly well understood, thanks to the pioneering work in the 1940s and 50s by Tutte, who proved the following theorem in \cite{Tutte}, 
bounding the order of a vertex-stabiliser: 
\begin{theorem}  \label{thm:tutte}
  Let $\Gamma$ be a connected finite symmetric cubic graph. 
  Then $\mathrm{Aut}(\Gamma)$ acts regularly on the set of $s$-arcs of $\Gamma,$ for some $s \in \{1,2,3,4,5\}$.
\end{theorem}
In particular, the Orbit-Stabiliser Theorem gives $|\mathrm{Aut}(\Gamma)| = |V(\Gamma)|\,3\cdot2^{s-1}$ for some $s \le 5$, 
with vertex-stabiliser of order $3\cdot2^{s-1} = 3,6,12,24$ or $48$. 

Tutte's work was taken further by Djokovi\'{c} and Miller, who in \cite{DjokovicMiller} divided the collection of all connected finite symmetric cubic graphs 
into seven classes, depending on the value of $s$ and on the existence or otherwise of an involutory automorphism that reverses an arc. 
They proved that if $s=1,3$ or $5$, then such an automorphism always exists, but could not do the same when $s=2$ or $4$. 
Examples admitting no such involutory automorphism were constructed by Conder and Lorimer in  \cite{ConderLorimer}, 
thereby proving that all seven classes (one each for $s=1,3$ and $5$, and two each for $s = 2$ and $4$) are non-empty.  
These classes are now denoted by  $1$, $2^1$, $2^2$, $3$, $4^1$, $4^2$ and $5$,  
with $s$ or $s^1$ indicating the class of connected finite $s$-arc-regular cubic graphs admitting an involutory arc-reversing involution, 
and with $s^2$ indicating the class (for $s = 2$ or $4$) of such graphs admitting no such involutory automorphism. 
Sometimes these classes are denoted also by $1^{\prime}$, $2^{\prime}$,  $2^{\prime\prime}$,  $3^{\prime}$, $4^{\prime}$, $4^{\prime\prime}$ and $5^{\prime}$, 
respectively, in the obvious way. 

Moreover, Djokovi\'{c} and Miller showed in \cite{DjokovicMiller} that for each of the seven classes, the automorphism group 
of every graph $\Gamma$ in that class is a homomorphic image of particular `universal' group.  
This universal group is a free product $V *_A E$ of the pre-image $V$ of the stabiliser (in $\Aut(\Gamma)$)  
of a vertex $v$ and the pre-image $E$ of the stabiliser of an edge $e = \{v,w\}$ incident to $v$, with the pre-image $A$ of the stabiliser 
of the arc $(v,w)$ amalgamated.  The subgroups $V$, $E$ and $A$ are finite, and in each one of the seven classes, can be chosen independently of the given graph $\Gamma$. 
In fact $V$ is isomorphic to $C_3$, $S_3$, $S_3 \times C_2$, $S_4$ and $S_4 \times C_2$, for $s = 1,2,3,4$ and $5$, respectively, 
with $A$ being a Sylow $2$-subgroup (of index $3$) in each case, and $E$ contains $A$ as a subgroup of index $2$, together with
the pre-image of an arc-reversing automorphism of order $2$ or $4$. 

The seven universal groups were denoted by $A_1^{\prime}$, $A_2^{\prime}$,  $A_2^{\prime\prime}$,  $A_3^{\prime}$, 
$A_4^{\prime}$, $A_4^{\prime\prime}$ and $A_5^{\prime}$ in \cite{DjokovicMiller}, but we will denote them by 
$G_1$, $G_2^{\,1}$, $G_2^{\,2}$, $G_3$, $G_4^{\,1}$, $G_4^{\,2}$ and $G_5$, as in \cite{ConderLorimer}. 
Presentations for all seven groups are obtainable from \cite{DjokovicMiller}, and alternative (but equivalent) presentations for them are given in \cite{ConderLorimer}.

Furthermore, if $\theta$ is any smooth homomorphism from one of these universal groups $G_s$, $G_s^{\,1}$ or $G_s^{\,2}$ 
onto a finite group $G$, with `smooth' here meaning that $\theta$ preserves the orders of $V$, $E$ and $A$, 
then a connected finite symmetric cubic graph $\Gamma$ can be constructed with $G$ being an $s$-arc-regular subgroup of  $\Aut(\Gamma)$.  
Conditions ensuring that $G = \Aut(\Gamma)$ are given in \cite{ConderNedela}, based on observations made in any of \cite{ConderLorimer,DjokovicMiller,Goldschmidt}.  
The resulting correspondence between appropriate quotients of the universal group and members of the relevant class has enabled 
all graphs in each class to be determined up to  a certain order, extending the almost-complete collection of connected finite symmetric cubic graphs 
of order up to $512$ compiled by hand by Foster, now known as the Foster census \cite{FosterCensus}.   
With the help of various computational methods, this census was completed and also extended to order up to $768$ 
by Conder and Dobcs{\'a}nyi in \cite{ConderDobcsanyi}, and subsequently by the first author up to order $10000$ in \cite{Conder-SymmCubic10000}. 

The  Djokovi\'{c}-Miller classification was extended by Conder and Nedela  in \cite{ConderNedela}, showing that the collection of all 
finite symmetric cubic graphs can be divided into 17 classes, based on which types of arc-transitive group actions the graph admits. 
A subgroup of the automorphism group of such a graph $\Gamma$ may be said to act with type $s$ or $s^1$ if it acts regularly on the $s$-arcs 
of $\Gamma$ and contains an arc-reversing automorphism of order $2$, or with type $s^2$ if it so acts but without  an arc-reversing automorphism of order $2$.

For what follows later in this paper, there are two other important properties of finite symmetric cubic graphs that need to be mentioned.

\begin{proposition}\label{thm:2/3AR no 4/5AR}
If $G$ is an arc-transitive group of automorphisms of a connected finite symmetric cubic graph $\Gamma$ 
in class $4^1$, $4^2$ or $5$, then $G$ contains no subgroup acting with type $2^1$, $2^2$ or $3$. 
Equivalently, if the automorphism group of a finite symmetric cubic graph contains a $2$- or $3$-arc-regular subgroup, 
then it contains no $4$- or $5$-arc-regular subgroup.
\end{proposition}

\begin{proof}
This follows from observations made in \cite{ConderLorimer, DjokovicMiller,Goldschmidt}; see \cite[Corollary 2.2]{ConderNedela}. 
\end{proof}  
 
For the other property, we give the presentations for $G_2^{\,1}$, $G_2^{\,2}$ and $G_3$ from \cite{ConderLorimer}: 

\noindent 
$G_2^{\,1} = \langle\, h,p,a \ | \ h^3=a^2=p^2=1$, \ $apa=p$, $php=h^{-1} \,\rangle$,  

\noindent 
$G_2^{\,2} = \langle\, h,p,a \ | \ h^3=p^2=1$, \ $a^2=p$, \ $php=h^{-1} \,\rangle$, \ and 

\noindent 
$G_3 = \langle\, h,p,q,a \ | \ h^3=a^2=p^2=q^2=1$, \ $apa=q$, \ $qp=pq$, \ $ph=hp$, \ $qhq=h^{-1} \,\rangle$. 

\noindent
In these cases $V = \langle h,p \rangle$, $\, \langle h,p \rangle$ and $\langle h,p,q \rangle$, respectively, 
with $A = \langle p \rangle$, $\, \langle p \rangle$ and $\langle p,q \rangle$, respectively, 
and with $a$ being the pre-image of an automorphism reversing the arc $(v,w)$.

\begin{proposition}\label{prop:2AR subgroup in 3AR subgroup}
Let $G$ be a finite group that acts regularly on the $2$-arcs of the connected cubic graph $\Gamma$. 
Then $\Gamma$ has a 3-arc-regular group of automorphisms containing $G$ if and only if there is a 
group automorphism $\theta: G \to G$ which centralises the image of $V$ in $G$, and takes the image of $a$ to the image of $ap$. 
\end{proposition}

\begin{proof}
See \cite[Proposition 4.1]{ConderLorimer}, in which (by abuse of notation) $h$, $a$, $p$ and $H$ are the images 
in $G$ of the generators $h$, $a$ and $p$ of $G_2^{\,1}$ or $G_2^{\,2}$ and the subgroup $V$.  
\end{proof}  

Next we describe some important background on regular and chiral polytopes, helpful for what follows.  
More detailed information can be found on them in \cite{McMullenSchulte} and \cite{SchulteWeiss}. 

\section{Abstract polytopes and their automorphisms}
\label{sec:polytopes}

\subsection{Basic definition and properties}
\label{sec:polytopebasics}

Abstract polytopes capture the combinatorial properties of classically defined and studied geometric polytopes. 
Their definition is somewhat complicated, but not difficult.

An {\em abstract $n$-polytope} (or {\em abstract polytope of rank\/} $n)$ is a partially ordered set $\mathcal{P}$, 
endowed with a strictly monotone rank function having range $\{-1,0,1,...,n\}$. 
The elements of $\mathcal{P}$ are called {\em faces}, with those of rank $j$ called {\em $j$-faces},  
and the faces of ranks $0,1$ and $n-1$ usually called the {\em vertices}, {\em edges\/} and {\em facets\/} 
of $\mathcal{P},$ respectively. 
The poset $\mathcal{P}$ must satisfy the following conditions, which ensure that abstract polytopes 
share many of the combinatorial properties frequently associated with the face lattices of convex polytopes.

First, $\mathcal{P}$ has a minimum face, $F_{-1}$, and a maximum face $F_n$, and each maximal chain 
(called a {\em flag\/}) contains exactly $n + 2$ faces. 
Two flags are said to be {\em adjacent} if they differ by just one face, and if they differ only 
in their $i$-faces, they are said to be {\em $i$-adjacent}. 
Also $\mathcal{P}$ must be {\em strongly flag connected}, which means that any two flags $\Phi$ and $\Psi$ of $\mathcal{P}$ 
can be linked by a sequence of flags $\Phi = \Phi_0, \Phi_1, \dots, \Phi_k = \Psi$, all containing $\Phi \cap \Psi$, 
such that each two successive flags $\Phi_{i-1}$ and $\Phi_i$ are adjacent. 
Finally, $\mathcal{P}$ must have a certain homogeneity property called the {\em diamond condition}, 
namely that whenever $F$ and $G$ are faces of ranks $j - 1$ and $j + 1$ for some $j$ with $F \le G$, 
there are exactly two faces $H$ of rank $j$ such that $F\leq H\leq G$.
 
Note that by the diamond condition, for each flag $\Phi$ of $\mathcal{P}$ and each $i \in \{0, \dots, n-1 \}$ 
there exists a unique flag that is $i$-adjacent to $\Phi$. We denote such a flag  by $\Phi^i$.
 
Next, if $F$ and $G$ are faces of $\mathcal{P}$  with $F \leq G$, then the set $\{H \in \mathcal{P}\mid F\leq H \leq G\,\}$ 
is called a {\em section\/} of $\mathcal{P}$, and denoted  by $G/F$. In particular, each facet $F_{n-1}$ 
may be viewed as the section $F_{n-1}/F_{-1} = \{H \in \mathcal{P} \mid H \leq F_{n-1}\,\}$.  
Also if $F_0$ is a vertex, then the section $F_n/F_0 = \{H \in \mathcal{P} \mid F_0\leq H\,\}$ is 
called the {\em vertex-figure of $\mathcal{P}\!$ at $F_0$}.  

Note that every section $G/F$ of $\mathcal{P}$ is again an abstract polytope.
In particular, every rank 2 section $G/F$ between an $(i \shortm 2)$-face $F$ 
and an incident $(i + 1)$-face $G$ of an abstract polytope $\mathcal{P}$ is isomorphic 
to the face lattice of a (possibly infinite) polygon. 
If for each $i$ the number of sides of this polygon is a constant, depending only on $i$ and not on $F$ or $G$, 
then we say that $\mathcal{P}$ is an {\em equivelar} polytope. 
Also we define the Schl\"afli type of an equivelar $n$-polytope $\mathcal{P}$ to be $\{p_1, p_2, \dots, p_{n-1}\}$, when
each section between an $(i-2)$-face and an $(i+1)$-face is an abstract $p_{i\,}$-gon. 
The Platonic solids are equivelar polytopes, and have Schl\"afli types $\{3,3\}$, $\{4,3\}$, $\{3,4\}$, $\{3,5\}$, and $\{5,3\}$.

\medskip
We now consider symmetries, dualities and polarities of abstract polytopes. 

An \textit{automorphism} of a polytope $\mathcal{P}$ is an order- and rank-preserving bijection from $\mathcal{P}$ to $\mathcal{P}$, 
and the group of all automorphisms of $\mathcal{P}$ is denoted by $\Aut(\mathcal{P})$, or often by $\Gamma(\mathcal{P})$. 
Note that by the diamond condition and the strongly flag connected property, every automorphism is uniquely determined 
by its effect on any given flag.  

A \textit{duality} $\delta$ is a bijection between two polytopes $\mathcal{P}$ and $\mathcal{Q}$ which preserves incidence but reverses order, so that $F\leq G$ in $\mathcal{P}$ if and only if $F^\delta \geq G^\delta$ in $\mathcal{Q}$. If such a bijection exists, then $\mathcal{P}$ and $\mathcal{Q}$ are \textit{duals} of each other. Every polytope $\mathcal{P}$ has a unique \textit{dual} $\mathcal{P}^*$, and if $\mathcal{P} \cong \mathcal{P}^*$, we say that $\mathcal{P}$ is \textit{self-dual}. If $\mathcal{P}$ is equivelar with type $\{p_1,\ldots,p_{n-1} \}$, then $\mathcal{P}^*$  is equivelar with type $\{p_{n-1},\ldots,p_1\}$, so if $\mathcal{P}$ is self-dual, then $p_i=p_{n-i}$ for $1 \leq i < n$.

Some self-dual polytopes admit a \textit{polarity} (a duality of order 2). 
In such cases we call the group of all automorphisms and dualities of $\mathcal{P}$ the \textit{extended group} $\overline{\Gamma}(\mathcal{P})$ of $\mathcal{P}$, 
and this contains the automorphism group $\Gamma(\mathcal{P})$ as a subgroup of index 2. 
Also if such a polytope $\mathcal{P}$ admits a polarity that preserves the orbit of every flag of $\mathcal{P}$ under $\Gamma(\mathcal{P})$, 
then $\mathcal{P}$ is said to be \textit{properly self-dual}, while otherwise it is called \textit{improperly self-dual}. 
Note that if $\mathcal{P}$ is properly self-dual, then all of its dualities preserve the flag orbits.
    
\medskip
In the next few subsections we describe two important classes of highly symmetric polytopes, namely regular and chiral polytopes.

\subsection{Regular polytopes}
\label{subsec:regular polytopes}

An abstract $n$-polytope $\mathcal{P}$ is called \textit{regular} if its automorphism group $\Gamma(\mathcal{P})$ acts transitively on the flags of  $\mathcal{P}$.  
As every automorphism is uniquely determined by its effect on any given flag, this definition implies that regular polytopes are 
precisely those whose automorphism group has maximum possible order (equal to the number of flags).

Now let $\mathcal{P}$ be any regular $n$-polytope.
Then every section of $\mathcal{P}$ is a regular polytope, with automorphism group a subgroup of $\Gamma(\mathcal{P})$. 
Moreover, any two sections of $\mathcal{P}$ defined by faces of the same ranks are isomorphic, and so $\mathcal{P}$ is equivelar. 
Let $\{p_1,\ldots,p_{n-1}\}$ be its type, and also set some flag $\Phi = \{F_{-1},F_0,\ldots,F_n \}\in \mathcal{F}(\mathcal{P})$ 
as the \textit{base flag} of $\mathcal{P}$.

As $\Gamma(\mathcal{P})$ acts regularly on the flags of $\mathcal{P}$, it follows that for $i=0,1,\ldots,n-1$ we can define $\rho_i$ to be 
the unique automorphism sending the base flag $\Phi$ to its $i$-adjacent flag $\Phi^i$. 
The resulting $n$ automorphisms are involutions, and are often referred to as `reflections' of the polytope $\mathcal{P}$, 
in the same way that a reflection of the $3$-cube about an edge $e = \{v,w\}$ fixes the vertex $v$ and the edge $e$ 
while interchanging the two faces incident with $e$.

Moreover, as shown in \cite{McMullenSchulte}, the $n$ involutions $ \rho_0,\ldots,\rho_{n-1}$ generate the automorphism group of $\mathcal{P}$, 
and satisfy (at least) the Coxeter relations\\[-12pt] 
    \begin{equation}
       \rho_i^2 = 1 \ \text{ for }  0 \le i < n \quad  \text{and} \quad (\rho_i\rho_j)^{m_{ij}}=1 \ \text{where} \ 
        m_{ij} =    \begin{cases} 
                        1 & \hbox{if }\  i=j \\
                        2 & \hbox{if }\  | i-j | \geq 2 \quad {}\\
                        p_i & \hbox{if }\  j=i+1
                    \end{cases}
        \label{equ:relns of rhos}
    \end{equation}
as well as the following \textit{intersection conditions}:\\[-8pt] 
    \begin{equation}
        \langle \, \rho_i \mid i \in I \, \rangle \cap \langle \, \rho_i \mid i \in J \, \rangle     = \langle \, \rho_i \mid i \in I \cap J \, \rangle \quad
        \text{for all }I,J \subseteq \{0,\ldots,n-1\}. \quad 
        \label{equ:int_conds_reg}
    \end{equation}
    
For later use, it is also helpful to understand the stabilisers of faces of $\mathcal{P}$ in $\Gamma(\mathcal{P})$. 
As is also shown in \cite{McMullenSchulte}, the stabiliser in $\Gamma(\mathcal{P})$ of the $i$-face $F_i$ of the base flag $\Phi$ of $\mathcal{P}$ is \\[-8pt] 
    \begin{equation}
      \Gamma(\mathcal{P})_{F_i} = \langle \, \rho_j  :  0 \le j < n \mid  j\ne i  \, \rangle. \qquad
        \label{equ:stabs_reg}
    \end{equation}
    
\smallskip
Next, for $i=1,\ldots,n-1$ we define $ \sigma_i = \rho_{i-1}\rho_i$, and call this a `rotation',  
as $\sigma_i$ fixes the faces in $\Phi \setminus \{F_{i-1},F_i\}$ and cyclically permutes both the $i$-faces and the $(i-1)$-faces in the section $F_{i+1} / F_{i-2}$. 
These $n-1$ rotations generate the \textit{rotation subgroup} $\Gamma^+ (\mathcal{P})$ of $\Gamma(\mathcal{P})$, which has index at most 2 in $\Gamma(\mathcal{P})$,  
 and it follows from Equation (\ref{equ:relns of rhos}) that they satisfy the following relations:
    \begin{equation}
        \sigma_j^{p_j} = 1 \ \hbox{ for } 1 \leq j \leq n-1 
         \quad  \hbox{and} \quad 
        (\sigma_j \sigma_{j+1} \ldots \sigma_k )^2 = 1 \ \hbox{ for } 1 \leq j<k\leq n-1.    \ 
        \label{equ:sigma_relns}
    \end{equation}

We say that $\mathcal{P}$ is \textit{directly-regular} if the index of $\Gamma(\mathcal{P})^+$ in $\Gamma(\mathcal{P})$ is 2, and \textit{non-orientably regular} otherwise. 
All `classical' regular polytopes (such as polygons and the Platonic solids) are directly-regular, while the hemi-cube (obtained by identifying antipodally opposite pairs of vertices of the cube) is non-orientably regular. 

Finally, note that a regular polytope $\mathcal{P}$ is self-dual if and only if there exists a (unique) polarity $\delta\in \overline{\Gamma}(\mathcal{P})$ 
that preserves the base flag $\Phi$ but reverses all of its faces. Such a polarity induces an involutory group automorphism of $\Gamma(\mathcal{P})$ 
taking $\rho_j$ to $\rho_{n-1-j}$ for all $j$, and this implies that $\overline{\Gamma}(\mathcal{P})$ is isomorphic to the semi-direct product $\Gamma(\mathcal{P}) \rtimes C_2$. 

\subsection{Chiral polytopes}
\label{subsec:chiral polytopes}

An abstract $n$-polytope $\mathcal{P}$ is said to be \textit{chiral} if its automorphism group $\Gamma(\mathcal{P})$ has two flag orbits, 
with two adjacent flags always lying in different orbits.
Such polytopes have maximal rotational symmetry but admit no reflections, and come in pairs, 
with  each member of a pair being a `mirror image' of the other (as the name suggests). 

Chiral polytopes are neither as easy to find nor as easy to study as regular polytopes.  
In particular, there are no finite chiral polytopes in Euclidean 3-space, and no convex chiral polytopes.
There are plenty of rank $3$ (viewable also as orientably-regular but chiral maps), but it becomes more challenging to 
construct examples as the rank increases. 

Indeed for quite some time, the only known finite examples of chiral polytopes had ranks $3$ and $4$, but just over 15 years ago, 
some finite examples of rank $5$ were constructed by Conder, Hubard and Pisanski \cite{ConderHubardPisanski}, 
and these were followed by others of ranks $5$ to $8$ by Conder and Devillers (in unpublished work). 
Soon after, Pellicer \cite{Pellicer} proved the existence of finite chiral polytopes of every rank $d\geq3$.

Pellicer's proof involved a clever construction, but this produced examples that could have `extremely large' order. 
More recently, Conder, Hubard and O'Reilly-Regueiro \cite{ConderHubardOReilly-Regueiro} showed how to construct concrete examples 
of chiral $n$-polytopes of type $\{3,3,\dots,3,k\}$ for some $k$, as extensions of regular simplices, 
with all but finitely many alternating and symmetric groups as their automorphism groups (for every integer $n \ge 5$). 
Meanwhile Cunningham \cite{Cunningham} gave an explanation for the large orders, proving that for $n \ge 8$, 
every chiral $n$-polytope has at least $48(n-2)(n-2)!$ flags. 

\smallskip
We now give some of the basic theory and properties of chiral polytopes, as can be found in more detail in \cite{McMullenSchulte} and \cite{SchulteWeiss}. 

\smallskip
Let $\mathcal{P}$ be any chiral $n$-polytope, where $n\geq 3$. Then for any given base flag $\Phi=\{F_{-1},\ldots,F_n\}$ of $\mathcal{P}$, 
there exist automorphisms $ \sigma_1, \ldots, \sigma_{n-1}$ of $\mathcal{P}$ such that each $\sigma_i$ fixes the faces in $\Phi \setminus \{F_{i-1},F_i \}$, 
and cyclically permutes the $i$-faces and the $(i-1)$-faces in the section $F_{i+1} / F_{i-2}$. 
Moreover, these automorphisms generate $\Gamma(\mathcal{P})$, and can be chosen so that $\sigma_i$ takes $\Phi$ to $(\Phi^{i})^{i-1}$, for $1 \le i < n$, 
and also  $\mathcal{P}$ is equivelar of type $\{p_1,\ldots,p_{n-1}\}$, where each $p_i$ is the order of the `rotation' $\sigma_i$. 
Again in this case the rotation group $\Gamma(\mathcal{P})^+$ of $\mathcal{P}$ is the subgroup generated by $\sigma_1, \ldots, \sigma_{n-1}$, 
and so $\Gamma(\mathcal{P})^+=\Gamma(\mathcal{P})$.
    
These $n-1$ rotations satisfy the relations in (\ref{equ:sigma_relns}), and they also satisfy certain intersection conditions 
analogous to those in (\ref{equ:int_conds_reg}), as follows. For $0 \leq i \leq j \leq n$, define
    \begin{equation*} 
    \tau_{i,j} =
        \begin{cases}
            1 & \text{if $\ 0 = i < j\ $ or $\ i < j = n$,} \\
            \sigma_i & \text{if $\ 0<i=j<n$,}\\
            \sigma_i \sigma_{i+1} \ldots \sigma_{j-1} \sigma_j & \text{otherwise.}
        \end{cases} 
    \end{equation*}
    Then for all subsets $I$ and $J$ of $ \{-1,0,\ldots,n\}$, these automorphisms satisfy 
    \begin{equation}
        \langle \, \tau_{s+1,t} \mid s,t\in I \, \rangle \cap 
        \langle \, \tau_{s+1,t} \mid s,t \in J \, \rangle =
        \langle \, \tau_{s+1,t} \mid s,t \in I \cap J \, \rangle. \quad {}
        \label{equ:int_conds_chi}
    \end{equation}

It is also easy to prove that every section of a chiral polytope is itself a regular or chiral polytope.

\smallskip
The stabiliser in $\Gamma(\mathcal{P})$ of the face $F_i$ of the chosen flag $\Phi$ of $\mathcal{P}$ is given by
    \begin{equation}
    \Gamma(\mathcal{P})_{F_i}= 
        \begin{cases}
           \langle \sigma_2,\sigma_3,\ldots,\sigma_{n-1} \rangle & \hbox{for }\,i=0, \\
           \langle \{ \sigma_j \mid j \not\in \{i,i+1\} \} \cup \{ \sigma_i \sigma_{i+1} \} \rangle & \hbox{for }\,1 \le i \le n-2, \\
           \langle \sigma_1,\sigma_2,\ldots, \sigma_{n-2} \rangle & \hbox{for }\,i=n-1.
        \end{cases}
    \label{equ:stabs_chi}
    \end{equation}

Each of these coincides (in its description) with the stabiliser of the corresponding face in the rotation subgroup for a regular polytope. 
In fact there is quite a close relationship between the structure of the rotation subgroup of a directly-regular polytope (with $\Gamma(\mathcal{P})^+ \ne \Gamma(\mathcal{P})$), 
and the automorphism group of a chiral polytope, in that each group is generated by the `rotations' $\sigma_i$.  The principle difference is the existence or otherwise of reflections, 
as summarised in the following theorem from \cite{SchulteWeiss}, in which the effect of the group automorphism $\varphi$ on the generators $\sigma_i$ is the same as conjugation by  $\rho_0$ in the regular case: 
    
\begin{theorem}
\label{thm:whenchiral}
If $A$ is a finite group generated by $n-1$ elements $\sigma_1,\ldots,\sigma_{n-1}$ which satisfy the relations in {\em (\ref{equ:sigma_relns})} and intersection conditions in {\em (\ref{equ:int_conds_chi})}, then $A$ is isomorphic to the rotation subgroup $\Gamma(\mathcal{P})^+$ of some chiral or directly-regular polytope $\mathcal{P}$. Moreover, $\mathcal{P}$ is directly-regular if and only if there exists an involutary group automorphism $\varphi: A \to A$ taking $\sigma_1$ to $\sigma_1^{-1}$, 
and $\sigma_2$ to $\sigma_1^{\,2}\sigma_2$, and $\sigma_i$ to $\sigma_i$ for $3 \le i < n$. 
\end{theorem}
    
Finally in this subsection, we consider self-duality.
Every self-dual regular polytope is properly self-dual, but in contrast, chiral polytopes can be properly or improperly self-dual, as defined in Subsection \ref{sec:polytopebasics}. 
By a pair of theorems in \cite{HubardWeiss}, the kind of self-duality is determined by the following theorem. 

 \begin{theorem}        
Let $\mathcal{P}$ be a chiral polytope with distinguished generators $\{\sigma_1,\ldots,\sigma_{n-1}\}$ with respect to its base flag $\Phi$. 
Then \\[+4pt] 
{\rm (a)} $\mathcal{P}$ is properly self-dual if and only if there exists an involutory automorphism of $\Gamma(\mathcal{P})$ 
      \\ ${}$ \quad \ taking $\sigma_i$ to $\sigma_{n-i}^{-1}$ for all $i \in \{1,\ldots,n-1\}$, while  \\[+4pt] 
{\rm (b)} $\mathcal{P}$ is improperly self-dual if and only if there exists a group automorphism of $\Gamma(\mathcal{P})$ 
      \\ ${}$ \quad \ taking $\sigma_i$ to $\sigma_{n-i}^{-1}$ for $1 \le i \le n-3$, and $\sigma_{n-2}$ to $\sigma_1 \sigma_2 \sigma_1^{-1}$, and $\sigma_{n-1}$ to  $\sigma_1$.

\label{thm:ISD or PSD chiral pol}
\end{theorem}
    

\subsection{Constructing directly-regular or chiral polytopes}
\label{subsec:constructingpolytopes}

There are many methods for constructing directly-regular or chiral polytopes, and we briefly describe just three of them here.

The first method, which we use to construct improperly self-dual chiral polytopes in Section \ref{sec:ISD chiral 3q3},
takes quotients of the orientation-preserving subgroup (of index $2$) in an infinite string Coxeter group, as was done in \cite{ConderHubardPisanski}. 

A {\em Coxeter group} is a finitely-presented group with generators $x_1,\ldots,x_n$ subject to the relations $(x_i x_j)^{m_{ij}} = 1$ for $1\le i \le j \le n$, 
with $m_{i,i}=1$ (so that $x_i^{\,2} = 1$) for all $i$, and with $m_{i,j} \geq 2$ whenever $i \neq j$. 
If also $m_{i,j}=2$ whenever $| i-j | \geq 2$, then this is called a \textit{string Coxeter group}; and if $p_i=m_{i,i+1}$ for $1 \le i < n$,  
then the group is often denoted by $[p_1,\ldots,p_{n-1}]$, and illustrated by the Coxeter-Dynkin diagram in Figure \ref{fig:dynkin_donuts}.

    \begin{figure}[h]
        \centering
        \includegraphics[width=0.45\textwidth]{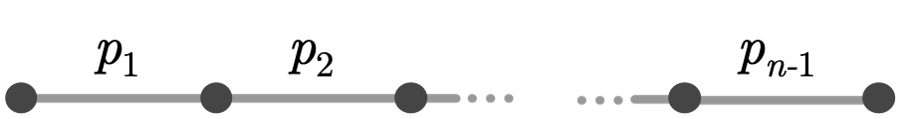}
        \caption{Coxeter-Dynkin diagram for the string Coxeter group $[p_1,\ldots,p_{n-1}]$.}
        \label{fig:dynkin_donuts}
    \end{figure}

Now let $\Lambda$ be the string Coxeter group  $[p_1,\ldots,p_{n-1}]$, and let $\Lambda^+$ be its `rotation subgroup', 
namely the index $2$ subgroup consisting of all elements of $\Lambda$ expressible as words of even length in the canonical generators of $\Lambda$. 

Then the automorphism group $\Gamma(\mathcal{P})$ of every regular polytope $\mathcal{P}$ with Schl{\"a}fli type $\{p_1,p_2,\dots,p_{n-1}\}$ 
is a homomorphic image of $\Lambda$, with its rotation group $\Gamma(\mathcal{P})^+$ being the image of $\Lambda^+$ .
Conversely, if $A$ is any smooth homomorphic image of $\Lambda$ (meaning that the orders of the images of the canonical generators $x_i$ of $\Lambda$ 
and their pairwise products $x_i x_j$ are preserved), and the images $\rho_{i-1}$ of the $x_i$ satisfy the intersection conditions in (\ref{equ:int_conds_reg}), 
then there exists some regular $n$-polytope $\mathcal{P}$ with Schl{\"a}fli type $\{p_1,p_2,\dots,p_{n-1}\}$ such that $A$ is isomorphic to $\Gamma(\mathcal{P})$.

Similarly, the automorphism group $\Gamma(\mathcal{P})$ of every chiral polytope $\mathcal{P}$ with Schl{\"a}fli type $\{p_1,p_2,\dots,p_{n-1}\}$ 
is a homomorphic image of $\Lambda^+$, and conversely, if $A$ is any smooth homomorphic image of $\Lambda^+$, and the images $\sigma_i$ 
of the generators $x_i x_{i+1}$ 
satisfy the intersection conditions given 
in subsection~\ref{subsec:chiral polytopes} (in terms of the alternative generators $\tau_{i,j}$), then there exists some chiral 
or directly-regular $n$-polytope $\mathcal{P}$ with Schl{\"a}fli type $\{p_1,p_2,\dots,p_{n-1}\}$ such that $A$ is isomorphic to $\Gamma(\mathcal{P})^+$.
The polytope $\mathcal{P}$ is directly-regular if the epimorphism $\theta$ from $\Lambda^+$ to $A$ extends to an epimorphism from $\Lambda$ 
to $\Gamma(\mathcal{P})$ as in Theorem \ref{thm:whenchiral}, and otherwise it is chiral.  In the latter case, $\ker\theta$ is normal in $\Lambda^+$ but not in $\Lambda$. 

The command \texttt{LowIndexNormalSubgroups} in {\sc Magma} can be used to find all normal subgroups of up to a given index (within limits) 
in a finitely-presented group, such as a string Coxeter group $\Lambda$ or its rotation group $\Lambda^+$, 
and hence to find the automorphism groups of all regular and chiral polytopes of small rank with up to a given number of flags.  
Such a process uses checks on the intersection conditions and the extendability of the normality (or otherwise) of the normal subgroup of $\Lambda^+,$  
as done previously by the first author to produce the lists in \cite{Conder-Chiral,Conder-Regular}. 

Once the automorphism group $\Gamma(\mathcal{P})$ of a regular or chiral polytope $\mathcal{P}$ is found, we can use it to construct the polytope itself. 
This can be achieved by identifying the $j$-faces of $\mathcal{P}$ with the right cosets of $\Gamma(\mathcal{P})_{F_j}$ in $\Gamma(\mathcal{P})$ for all $j$, 
with the help of (\ref{equ:stabs_reg}) or (\ref{equ:stabs_chi}), and then the partial ordering can be found via non-empty intersection of cosets; see \cite{MonsonWeiss-2007}. 
  
\medskip
The second method involves using connected finite 2- or 3-arc-regular cubic graphs to construct 4-polytopes of type $\{3,q,3\}$. 
This is described in detail in \cite{MonsonWeiss-2007}, and we give only a brief summary for the  3-arc-regular case, with the 2-arc-regular case being similar. 

If the connected cubic graph $\Gamma $ is 3-arc-regular, then for each vertex $v$ there exists a unique non-trivial automorphism which fixes $v$ and its three neighbours.  
(In fact this automorphism is the image of the generator $p$ in the universal group $G_3$ for 3-arc-regular cubic graphs given in subsection \ref{subsec:cubic graphs}.)  

Now let $\rho_2$, $\rho_0$, $\rho_3$ and $\rho_1$ be such automorphisms for the four vertices in a given 3-arc $(v_0,v_1,v_2,v_3)$ of $\Gamma$, 
and let $\delta$ be the unique involutory automorphism of $\Gamma$ that reverses this 3-arc. 
If $\rho_0,$ $\rho_1$, $\rho_2$ and $\rho_3$ satisfy the intersection conditions in (\ref{equ:int_conds_reg}), then 
the group generated by these four automorphisms and $\delta$ is isomorphic to the extended automorphism group of a regular polytope $\mathcal{P}$ of type $\{3,q,3\}$. 
(Moreover, the medial layer graph (as defined in the next section) of this polytope $\mathcal{P}$ is isomorphic to either $\Gamma$ or the canonical double cover of $\Gamma$, depending on whether or not $\Gamma$ is bipartite.) 

\medskip
The third method constructs regular or chiral polytopes as covers of smaller polytopes. 

For example, if $\mathcal{Q}$ is a non-orientably regular $n$-polytope, then there exists a directly-regular $n$-polytope $\mathcal{P}$ 
that is the orientable double cover of $\mathcal{Q}$, constructible as follows. 
The rotation group $\Gamma(\mathcal{Q})^+$ of $\mathcal{Q}$ is the epimorphic image of the rotation subgroup $\Lambda^+$ of some string Coxeter group $\Lambda$, 
and by non-orientability $\Gamma(\mathcal{Q})^+$ is equal to $\Gamma(\mathcal{Q})$, and hence contains images of the $n$ reflecting generators of $\Lambda$. 
Now if $\alpha$ is the involutory automorphism of $\Gamma(\mathcal{Q})^+$ induced by conjugation by the image of one such generator, 
then the semi-direct product $\Gamma(\mathcal{Q})^+ \rtimes \langle \alpha \rangle$ is the automorphism group of a directly-regular polytope $\mathcal{P}$ 
with twice as many flags as $\mathcal{Q}$, and with $\Gamma(\mathcal{P})^+ \cong \Gamma(\mathcal{Q})^+ = \Gamma(\mathcal{Q})$.
    
Another approach takes a regular or chiral $n$-polytope $\mathcal{P}$, with automorphism group $\Gamma(\mathcal{Q})$ 
obtainable as the quotient of some string Coxeter group $\Lambda$ or its rotation subgroup $\Lambda^+$ by a normal subgroup $J$, 
and then constructs a larger regular or chiral $n$-polytope $\mathcal{Q}$ that covers $\mathcal{P}$, via some normal subgroup $K$ of 
$\Lambda$ or its rotation subgroup $\Lambda^+$ such that $K$ is contained in $J$.  
This approach was taken successfully to find chiral covers of chiral polytopes in~\cite{ConderZhang}, 
and to find chiral covers of regular polytopes in~\cite{Zhang}, with abelian covering group $J/K$ in each case.

\section{Medial layer graphs}
\label{sec:mediallayergraphs}

If $\mathcal{P}$ is a polytope with even rank $n$, then the two `middle' layers of the Hasse diagram of $\mathcal{P}$ form the \textit{medial layer graph} $\mathcal{G}$ of $\mathcal{P}$. The vertices of this graph are the $(\frac{n-2}{2})$-faces and $(\frac{n}{2})$-faces of $\mathcal{P}$, with two vertices being adjacent whenever the corresponding two faces are incident in $\mathcal{P}$. 
Obviously this graph  $\mathcal{G}$ is bipartite, and if $\mathcal{P}$ is equivelar with type $\{p_1,\ldots,p_{n-1}\}$, 
then the vertices of the two parts of  $\mathcal{G}$ have valencies $p_{\frac{n-2}{2}}$ and $p_{\frac{n+2}{2}}$.  
    
For the rest of this section, we will focus on equivelar 4-polytopes of type $\{p,q,p\}$ that are either regular and self-dual, or chiral and properly self-dual. 

Let $\mathcal{P}$ be such a polytope, and let $\Phi=\{F_{-1},F_0,F_1,F_2,F_3,F_4\}$ be its base flag. 
Note that every  2-face of $\mathcal{P}$ is a $p$-gon, and there are $p$ of these around each 1-face of $\mathcal{P}$, 
 and in every 3-face of $\mathcal{P}$, there are $q$ $p$-gons around each 0-face.
    
Because $\mathcal{P}$ is regular or chiral, its rotation group $\Gamma(\mathcal{P})^+$ is generated by $\sigma_1$, $\sigma_2$ and $\sigma_3$, 
satisfying the relations $\sigma_1^{\,p}=\sigma_2^{\,q}=\sigma_3^{\,p}=(\sigma_1\sigma_2)^2=(\sigma_2\sigma_3)^2=(\sigma_1\sigma_2\sigma_3)^2=1$  
given by (\ref{equ:sigma_relns}). Moreover, as shown by Schulte and Weiss in \cite{SchulteWeiss}, the intersection conditions in (\ref{equ:int_conds_chi}) 
can be replaced simply by 
    \begin{equation}
        \langle \sigma_1 \rangle \cap \langle \sigma_2 \rangle = \{1\},  \quad
        \langle \sigma_2 \rangle \cap \langle \sigma_3 \rangle = \{1\} \quad \hbox{and} \quad 
        \langle \sigma_1,\sigma_2 \rangle \cap \langle \sigma_2,\sigma_3 \rangle = \langle \sigma_2 \rangle.
        \label{equ:int_conds_chi_n=4}
    \end{equation}

If $\mathcal{P}$ is properly self-dual, admitting a polarity $\delta$ that reverses the base flag $\Phi$,  
then $\delta$ satisfies the relations $\, \delta^2 = 1, \  \delta^{-1} \sigma_1 \delta = \sigma_3^{-1}\,$ and $\,\delta^{-1} \sigma_2 \delta = \sigma_2^{-1}$,  
and we may regard $\overline{\Gamma}(\mathcal{P})$ as a subgroup of $\Aut(\mathcal{G})$.

\medskip
Next, following an approach taken by Monson and Weiss in \cite{MonsonWeiss-2008}, we may construct a covering graph $\mathcal{C}$ of $\mathcal{G}$,  
which is a Cayley graph on the group $G$ generated by the $p$ polarities of $\mathcal{P}$ defined by setting $\,\delta_j = \sigma_1^{\,1-j} \delta \sigma_1^{\,j-1}\,$ for $1\leq j \leq p$.

Note that these $p$ polarities are distinct, as they take the $2$-face $F_2$ of $\Phi$ to the $p$ distinct 1-faces in the section $F_3/F_0$.
Some interactions between them and the rotations $\sigma_i$ are given in the following Lemma from \cite{MonsonWeiss-2008}, 
with subscripts treated mod $p$:  
    \begin{lemma} ${}$\\[+6pt] 
    \begin{tabular}{lll} 
        {\rm (a)} $\,\delta_{j+k}=\sigma_1^{-k}\delta_j\sigma_1^{\,k};$ 
        &  {\rm (b)} $\,\delta_j\delta_k = \sigma_1^{\,1-j}\sigma_3^{\,k-j}\sigma_1^{\,k-1};$ \\
        {\rm (c)} $\,\sigma_2^{\,2}=\delta_2\delta_3\delta_2\delta_1;$ 
        &  {\rm (d)} $\,\sigma_3\delta_j\delta_3^{-1}= \delta_1\delta_2\delta_{j+1}\delta_2\delta_1 ;$ \\
        {\rm (e)} $\,(\sigma_2\sigma_3)^{-1} \delta_j \sigma_2\sigma_3 = \delta_{3-j};$
        & {\rm (f)} $\,\sigma_2^{-1}\delta_j\sigma_2=\delta_1\delta_2\delta_{4-j}\delta_2\delta_1\,$ {\rm and} $\,\sigma_2\delta_j\sigma_2^{-1}= \delta_2\delta_3\delta_{4-j}\delta_3\delta_2$.
        \end{tabular}
        \label{lem:deltas_lemma}
    \end{lemma}
    
\noindent    
In particular, the group $G$ generated by the polarities $\delta_1,\delta_2,\ldots,\delta_p$ is a normal subgroup of the 
extended rotation group $\overline{\Gamma}(\mathcal{P})^+ = \langle \sigma_1,\sigma_2,\sigma_3,\delta \, \rangle$. 
The relationships between these and other subgroups of $\Aut(\mathcal{G})$ associated with $\mathcal{P}$ are illustrated in Figure \ref{fig:relns of subgroups}.
    \smallskip
    
    \begin{figure}[h]
        \centering
        \includegraphics[width=0.40\textwidth]{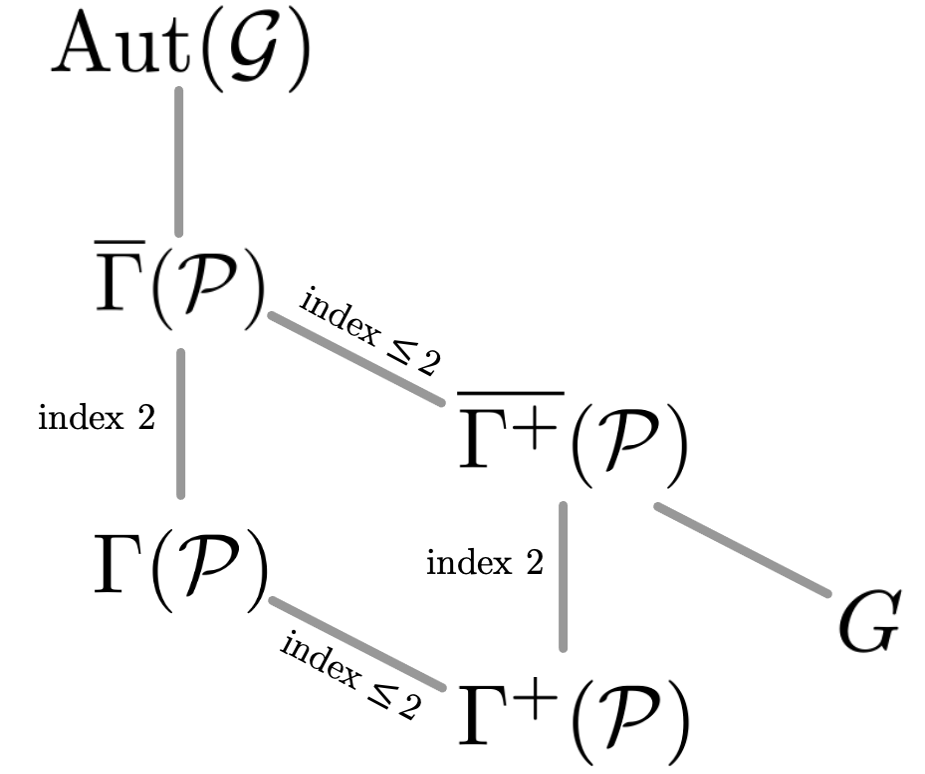}
        \caption{Inclusions between certain subgroups of $\mathrm{Aut}(\mathcal{G})$.}
        \label{fig:relns of subgroups}
    \end{figure}
                
    Now let $\mathcal{C}$ be the Cayley graph for $G$ with generating set $\{\delta_1,\ldots,\delta_p\}$, so that two vertices $g$ and $h$ are adjacent in $\mathcal{C}$ if and only if $gh^{-1}=\delta_i$ for some $i\in \{1,\ldots,p\}$. This graph is a connected symmetric $p$-valent  cover of $\mathcal{G}$, with the Cayley group $G$ acting transitively (indeed regularly) on its vertices, so every vertex of $\mathcal{G}$ can be written in the form $F_{2\hskip 1pt}g$ for some $g \in G$. 
   Also let $m=| G \cap \Gamma(\mathcal{P})^+_{F_2} |$, and recall that from equations (\ref{equ:stabs_reg}) and (\ref{equ:stabs_chi}) that the stabiliser of $F_2$ in $\Gamma(\mathcal{P})^+$ is given by
    \begin{equation}
    \Gamma(\mathcal{P})^+_{F_2}= \begin{cases}
        \langle \sigma_1,\sigma_2\sigma_3 \rangle & \text{if $\mathcal{P}$ is directly-regular or chiral, or} \\
        \langle \, \rho_0,\rho_1,\rho_3 \, \rangle & \text{if $\mathcal{P}$ is non-orientably regular.}
                                \end{cases}
    \end{equation}
    The first possibility $\langle \sigma_1,\sigma_2\sigma_3 \rangle$ is isomorphic to the dihedral group $D_p$ of order $2p$, 
    since $\sigma_1$ has order $p$ and both $\sigma_2\sigma_3$ and $\sigma_1\sigma_2\sigma_3$ are involutions.
    For the second possibility, where $\mathcal{P}$ is non-orientably regular, the subgroup of $\Gamma(\mathcal{P})$ generated by $\rho_0$ and $\rho_1$ is dihedral of order $2p$, 
    because the section $F_2/F_{-1}$ is a regular $p$-gon; see \cite[Proposition 2B9]{McMullenSchulte}. 
   Then since $\rho_3$ commutes with $\rho_0$ and $\rho_1$, it follows that 
    $\langle \, \rho_0,\rho_1,\rho_3 \, \rangle \cong D_p \times C_2$, of order $4p$. 
    
Next, we expand on the relationship between $\mathcal{C}$ and $\mathcal{G}$ by noting observations made in \cite{MonsonWeiss-2008}.  
Clearly there exists a surjective graph homomorphism $\nu\!: \mathcal{C} \to \mathcal{G}$ taking $g$ to $F_{2\hskip 1pt}g$ for all $g \in G$, 
and this is injective on the neighbourhood of every vertex of~$\mathcal{C}$. Moreover, if $\mathcal{P}$ is chiral or directly-regular, then $\mathcal{C}$ is a $s$-fold covering of $\mathcal{G}$, where $s=| G \cap \Gamma(\mathcal{P})^+_{F_2} |$, while if $\mathcal{P}$ is non-orientably regular, then $\mathcal{C}$ is a $2s$-fold covering of $\mathcal{G}$. 
Vertices $g$ and $h$ of $\mathcal{C}$ are taken to the same vertex in $\mathcal{G}$ under $\nu$ if and only if $gh^{-1} \in G \cap \Gamma(\mathcal{P})_{F_2}$. Also $1\leq s\leq 4p$, since $\Gamma(\mathcal{P})^+_{F_2}$ has order $2p$ when $\mathcal{P}$ is directly-regular or chiral, and order $4p$ when $\mathcal{P}$ is non-orientably regular. 

\medskip
Finally, we have the following important theorem from \cite[Theorems 2 and 5]{MonsonWeiss-2007}, which exhibits the  relationship between 
the properties of $\mathcal{P}$ and $\mathcal{G}$ when $p = 3$ (and $\mathcal{G}$ is cubic): 
    \begin{theorem}
    Let $\mathcal{P}$ be a finite self-dual regular or chiral polytope, with type $\{3,q,3\}$. 
    Then $\mathrm{Aut}(\mathcal{G})$ is isomorphic to the extended group $\overline{\Gamma}(\mathcal{P})$ of $\mathcal{P}$, 
     and the medial layer graph $\mathcal{G}$ is $3$-arc-regular when $\mathcal{P}$ is regular, and $2$-arc-regular when $\mathcal{P}$ is chiral.
    \label{thm:p=3 s-AR of medial layer graph}
    \end{theorem}

\section{Improperly self-dual chiral polytopes of type $\{3,\lowercase{q},3\}$}
\label{sec:ISD chiral 3q3}

One of the final comments made by Monson and Weiss in \cite{MonsonWeiss-2008} raised the question of 
existence of improperly self-dual chiral polytopes of type $\{p,q,p\}$ (when $q$ is even). 

We answer this question positively by giving two examples, namely two improperly self-dual chiral polytopes 
of types $\{3,6,3\}$ and  $\{3,18,3\}$.
We constructed these examples (with the help of {\sc Magma}) by finding quotients of the rotation subgroup of the string Coxeter group $[3,q,3]$ 
for different fixed values of $q$, using the first method described in subsection \ref{subsec:constructingpolytopes}, 
and using Theorem \ref{thm:ISD or PSD chiral pol} to check the resulting chiral polytopes to see if any were improperly self-dual. 
    
The automorphism group of the first example is a quotient $A$ of the rotation subgroup of the Coxeter group $[3,6,3]$ with order $|A| = 18522$, 
obtainable by adding the relations
$$  (\sigma_1^{-1}\sigma_2^{\,2})^4\sigma_1\sigma_2^{-2}
      =  (\sigma_2^{\,2}\sigma_3^{-1})^3\sigma_2^{\,2}\sigma_3\sigma_2^{-2}\sigma_3^{-1}
      =  (\sigma_2^{-1}\sigma_3\sigma_1\sigma_2^{-1}\sigma_3\sigma_2\sigma_1)^2=1.
$$
This polytope has $147$ vertices, $3087$ edges, $3087$ 2-faces, and $147$ 3-faces,  
and its medial layer graph is a 2-arc-regular bipartite cubic graph of order $6174$, with automorphism group of order $2|A| = 37044$. 
As pointed out by the referee, the 147 facets are isomorphic to the chiral toroidal map of type $\{3,6\}_{(4,1)}$, 
and dually, the vertex-figures are isomorphic to the one of type $\{6, 3\}_{(4,1)}$.
    
The automorphism group of the second example is a quotient $B$ of the rotation subgroup of the Coxeter group $[3,18,3]$ with order $|B| = 39366 = 2 \cdot 3^9$, 
obtainable by adding the relations \\[-8pt]
$$  \sigma_2^{\,5} \sigma_1\sigma_2^{-2}\sigma_1\sigma_2^{-1}\sigma_1\sigma_2^{-4}\sigma_1
   = (\sigma_2^{-1}\sigma_3\sigma_1\sigma_2^{-1}\sigma_3^{-1}\sigma_1\sigma_2^{-1}\sigma_1)^2    $$ \\[-36pt]  $$
   = \sigma_2^{\,2}\sigma_1^{-1}\sigma_3\sigma_2 \sigma_1\sigma_2^{-1}\sigma_3 \sigma_2^{\,2}\sigma_1^{-1}\sigma_3\sigma_2 \sigma_1^{-1}\sigma_2\sigma_1^{-1}=1.
$$
 
\noindent 
This polytope has $81$ vertices, $6561$ edges, $6561$ 2-faces, and $81$ 3-faces,  
and its medial layer graph is a 2-arc-regular bipartite cubic graph of order $13122$, with automorphism group of order $2|B| = 78732$. 
Also we note (following a suggestion by the referee) that the group $B$ has a non-central normal subgroup $K$ 
of order $9$ such that the quotient $B/K$ (of order $2 \cdot 3^7$) is the rotation group of a directly regular polytope of type $\{3,18,3\}$, 
covered by the given chiral polytope (with abelian covering group isomorphic to $C_3 \times C_3$).    

\medskip
Recall that self-dual regular polytopes and properly self-dual chiral polytopes admit a polarity that reverses the base flag, 
while all dualities of improperly self-dual chiral polytopes send the base flag to the reverse of a flag in the other flag orbit, 
and hence improperly self-dual chiral polytopes do not admit such a polarity.  They can still admit other kinds of polarity, and indeed it was 
mentioned by Schulte and Weiss in \cite{SchulteWeiss} that they did not know at the time of any examples 
of highly symmetric self-dual polytopes that do not admit a polarity.  

The two examples above do not admit any polarities at all.
(The first one admits $18522$ dualities, of which $6174$ have order $4$ and $12348$ have order $12$, 
while the second admits $39366$ dualities, of which $4374$ have order $4$, and $8748$ have order $12$, and $26244$ have order $36$.)  
Among the very few other such examples known are the two improperly self-dual chiral 4-polytopes with types $\{4,4,4\}$ and $\{9,6,9\}$ given in the first 
author's  list of chiral polytopes with few flags \cite{Conder-Chiral}, and the universal polytope of type $\{\{4,4\}_{(1,3)},\{4,4\}_{(1,3)}\}$ 
mentioned by Hubard and Weiss in \cite{HubardWeiss}.

\section{Arc transitivity of $\mathcal{G}$ and $\mathcal{C}$}
\label{sec: arc transitivity of G and C}

Monson and Weiss asked in \cite{MonsonWeiss-2008} if the Cayley graph $\mathcal{C}$ they associated 
with a self-dual regular or chiral polytope of type $\{p,q,p\}$ could have a higher degree of $s$-arc-transitivity than the 
medial layer graph $\mathcal{G}$ that it covers -- or in other words, if $\mathcal{G}$ is $t$-arc-transitive 
and $\mathcal{C}$ is $s$-arc-transitive, could $s$ be greater than $t$? 
They were particularly interested in the case $p=3$, where Theorem \ref{thm:tutte} implies that $s \leq5$, 
and Theorem \ref{thm:p=3 s-AR of medial layer graph} implies that $t = 2$ or $3$. They suspected that $s$ cannot be 4 or 5, and we will prove it. 
    
Suppose that $\mathcal{G}$ is $t$-arc-transitive and $\mathcal{C}$ is $s$-arc-transitive.
We will show that if the polytope is chiral, then $s=t=2$, and if the polytope is regular, then $2 \leq s \leq 3=t$. 
To do this, we construct a subgroup of $\mathrm{Aut}(\mathcal{C})$ that acts 2-arc-regularly on $\mathcal{C}$, 
and show that this has index at most $2$ in $\mathrm{Aut}(\mathcal{C})$, and that if this index is $2$ then $\mathcal{P}$ is regular. 
    
Recall from Section \ref{sec:mediallayergraphs} that the Cayley group $G$ of $\mathcal{C}$ is generated by $\{\delta_1,\delta_2,\delta_3\}$, 
and is normal in $\overline{\Gamma}(\mathcal{P})^+ = \langle \sigma_1,\sigma_2,\sigma_3, \delta \rangle$.
Again with a slight abuse of notation, let $h$ and $p$ be the group homomorphisms of $G$ defined by conjugation 
by $\sigma_1$ and $\sigma_1\sigma_2\sigma_3$, respectively, and let $a$ be the involutory automorphism of $\mathcal{C}$ 
defined by right multiplication by $\delta_1$ ($= \delta$).  

By Lemma \ref{lem:deltas_lemma} we see that $h$ induces a $3$-cycle on the Cayley set $\{\delta_1,\delta_2,\delta_3\}$ for $\mathcal{C}$, 
while $p$ induces a  $2$-cycle (fixing $\delta_1$), so $h$ and $p$ induce graph automorphisms of $\mathcal{C}$, fixing the vertex $1_G$. 
Moreover, as  $\sigma_1\sigma_2\sigma_3$ commutes with $\delta_1$ and conjugates $\sigma_1$ to $\sigma_1^{-1}$, 
we find that $p^a=p$ and $h^p=h^{-1}$. 
It follows that $h$, $p$ and $a$ satisfy the defining relations for $G_2^{\,1}$ from \cite{ConderLorimer} that were given after Proposition \ref{thm:2/3AR no 4/5AR}, 
and so they generate a subgroup $T$ of $\mathrm{Aut}(\mathcal{C})$ that acts 2-arc-regularly on $\mathcal{C}$, 
with vertex-stabiliser $\langle h,p \rangle \cong S_3$. 

Moreover, by Proposition \ref{thm:2/3AR no 4/5AR} this implies that $\mathrm{Aut}(\mathcal{C})$ does not contain a $4$- or $5$-arc-regular subgroup, 
and so $\mathcal{C}$ cannot be $4$- or $5$-arc-regular, and hence is $2$- or $3$-arc-regular. 
In particular, the index of $T$ in $\mathrm{Aut}(\mathcal{C})$ is at most $2$.
    
Now if $\mathcal{P}$ is regular, then we know that $\mathcal{G}$ is $3$-arc-regular, and therefore $\mathcal{C}$ cannot have 
a higher degree of $s$-arc-regularity than $\mathcal{G}$.

On the other hand, we can show that if $\mathcal{P}$ is chiral 
(which implies by Theorem \ref{thm:p=3 s-AR of medial layer graph} that $\mathcal{G}$ is 2-arc-regular), 
then $\mathcal{C}$ cannot be $3$-arc-regular. 
For if we assume the contrary,  
then by Proposition \ref{prop:2AR subgroup in 3AR subgroup}, there exists a group automorphism $\theta$ of $\langle h,p,a\rangle = T$ 
that fixes $h$ and $p$, and takes $a$ to $ap$. 
Accordingly, $\theta$ induces an automorphism $\psi$ of $\langle \sigma_1,\sigma_1\sigma_2\sigma_3,\delta \rangle $ 
that fixes $\sigma_1$ and $\sigma_1\sigma_2\sigma_3$, and takes $\delta$ to $\delta\sigma_1\sigma_2\sigma_3$.  
From this and the fact that $\delta \sigma_3 \delta = \sigma_1^{-1}$ we find that 
    $$ \delta_2^{\,\psi} = (\sigma_1^{-1}\delta\sigma_1)^\psi 
      = \sigma_1^{-1}(\delta\sigma_1\sigma_2\sigma_3)\sigma_1 
      =  \sigma_1^{-1}\delta(\sigma_1\sigma_2\sigma_3)\sigma_1 
      = \delta\sigma_3(\sigma_3^{-1}\sigma_2^{-1}\sigma_1^{-1})\sigma_1
      = \delta\sigma_2^{-1},$$
and then since $\sigma_3 = \delta \sigma_1^{-1} \delta = \delta \sigma_1^{-1} \delta \sigma_1 \sigma_1^{-1} = \delta \delta_2 \sigma_1^{-1}$, 
it follows that 
    $$\sigma_3^{\,\psi} = (\delta \delta_2 \sigma_1^{-1})^\psi 
      = (\delta\sigma_1\sigma_2\sigma_3)(\delta\sigma_2^{-1})(\sigma_1^{-1})
      = (\sigma_3^{-1}\sigma_2^{-1}\sigma_1^{-1})\sigma_2^{-1}\sigma_1^{-1} 
      = \sigma_3^{-1}(\sigma_1\sigma_2)^{-2} 
      = \sigma_3^{-1},$$
and so the $\psi$-image of $\sigma_2$ is 
    $$\sigma_2^{\,\psi} = (\sigma_1^{-1}(\sigma_1\sigma_2\sigma_3)\sigma_3^{-1})^\psi = \sigma_1^{-1} \sigma_1\sigma_2\sigma_3 \,\sigma_3=\sigma_2\sigma_3^{\,2}.$$
    
In particular, $\psi$ induces the same automorphism of $\Gamma(\mathcal{P})^+ = \langle \sigma_1,\sigma_2,\sigma_3 \rangle$ 
as would conjugation by the reflection $\rho_3$ in the case where $\mathcal{P}$ is regular, namely taking $(\sigma_1,\sigma_2,\sigma_3)$ 
to $(\sigma_1,\sigma_2\sigma_3^{\,2},\sigma_3^{-1})$. But the rotation subgroup of a chiral polytope admits no such automorphism, 
which is a contradiction.  It follows that $\mathcal{C}$ cannot be $3$-arc-regular, as claimed. 
    
Thus we have the following: 
     
\begin{theorem}
 \label{thm: s vs t for cubic}
Let $\mathcal{P}$  be a properly self-dual finite regular or chiral polytope with type $\{3,q,3\}$ for some $q$. 
If $\mathcal{G}$ is $t$-arc-regular and $\mathcal{C}$ is $s$-arc-regular, then $s \leq t$.
\end{theorem}

We complete this section by showing that both cases occur, by giving an example of a self-dual regular $4$-polytope $\mathcal{P}$ 
for which $\mathcal{C}$ is 2-arc-regular, and another for which $\mathcal{C}$ is 3-arc-regular.

The regular 4-simplex (or 5-cell) is a directly-regular polytope of type $\{3,3,3\}$ that has ten triangular 2-faces and five tetrahedral 3-faces, 
and automorphism group $S_5$ of order $120$.  Its medial layer graph $\mathcal{G}$ is isomorphic to the 3-arc-regular Desargues graph F020B 
(of order $20$) in the Foster census \cite{FosterCensus} and listed as C20.2 at \cite{Conder-SymmCubic10000},   
while the associated Cayley graph $\mathcal{C}$ is isomorphic to the 2-arc-regular graph F120B  
(of order $120$) in the Foster census \cite{FosterCensus} and listed as C120.2 at \cite{Conder-SymmCubic10000}.    

In contrast, the 24-cell (or octaplex) is a directly-regular polytope of type $\{3,4,3\}$ that has 96 triangular 2-faces and 24 octahedral 3-faces, 
and automorphism group of order $1152$.    Its medial layer graph $\mathcal{G}$ and the associated Cayley graph $\mathcal{C}$ 
are both isomorphic to the 3-arc-regular graph F192A in the Foster census \cite{FosterCensus} and listed as C192.3 at \cite{Conder-SymmCubic10000}. 
 
Curiously, the medial layer graph $\mathcal{G}$ and Cayley graph $\mathcal{C}$ for the 
properly self-dual chiral polytope of type $\{3,8,3\}$ listed at \cite{Conder-Chiral} are both isomorphic to another 
2-arc-regular graph of order $192$, namely the graph F192C  in \cite{FosterCensus} and listed as C192.2 at \cite{Conder-SymmCubic10000}. 
    
\section{Larger vertex-stabilisers for ${\mathcal C}$ than for ${\mathcal G}$}
\label{sec: larger stabiliser} 
    
The theory of 3-valent arc-transitive graphs is well-developed (as may be seen from its brief partial description in subsection~\ref{subsec:cubic graphs}, 
or in more detail in \cite{DjokovicMiller} or \cite{ConderLorimer}, for example), helped by the fact that Tutte's theorem (Theorem \ref{thm:tutte}) 
bounds the order of the vertex-stabilisers in the 3-valent case.  
For higher valencies it can be much more complicated, especially for non-prime valencies such as $4$ (for which the vertex-stabilisers have unbounded order). 
In particular, for such higher valencies not every arc-transitive finite graph is $s$-arc-regular for some $s$. 

Nevertheless a comparison of the automorphism groups of the medial layer graph $\mathcal{G}$ and its covering Cayley graph $\mathcal{C}$ 
for a regular or chiral $4$-polytope of type $\{p,q,p\}$ can be made in terms of the orders of their vertex-stabilisers. 
In the case $p = 3$ we know that this order is no greater for $\mathcal{C}$ than for $\mathcal{G}$, by Theorem \ref{thm: s vs t for cubic}, 
because those orders are $3 \cdot 2^{s-1}$ and $3 \cdot 2^{t-1}$ with $s \le t$. 
But we have an infinite family of examples with $p = 6$ for which the analogous property does not hold, constructible as follows. 

Let $\Lambda$ be the group obtainable from any string Coxeter group $[6,q,6]$ by deleting the relation $(x_2 x_3)^q = 1$, 
and adding the two extra relations  $((x_1 x_2)^2 x_3)^2  = ((x_3 x_4)^2 x_2)^2  = 1$.  

In this group, the element $(x_2 x_3)^3$ generates a normal subgroup $K$ of index $216$, 
which by the Reidemeister-Schreier process (implemented as the {\tt Rewrite} command in {\sc Magma})  
is free of rank $1$ and hence is infinite cyclic.  Moreover, $\Lambda/K$ is a smooth quotient of $\Lambda$, 
satisfying the intersection conditions (\ref{equ:int_conds_reg}), and in which the image of $x_2 x_3$ has order $3$.  
It follows that $\Lambda/K$ is the automorphism group of a regular polytope with type $\{6,3,6\}$, 
which (as noted by the referee) is the universal polytope $\{\{6,3\}_{(1,1)},\{3,6\}_{(1,1)}\}$ with $216$ flags 
given in~\cite[11C14]{McMullenSchulte}.

Now for any multiple of $3$, say $q = 3k$, let $N_q$ be the subgroup of $\Lambda$ generated by $(x_2 x_3)^q$. 
Then $N_q$ is characteristic in $N$ and therefore normal in $\Lambda$, with smooth quotient $\Lambda/N_q$ of order $216k = 72q$. 
It is not difficult to prove that $\Lambda/N_q$ satisfies the intersection conditions (\ref{equ:int_conds_reg}), and hence is the 
automorphism group of a directly-regular polytope $\mathcal{P}_q$ of type $\{6,q,6\}$. Also it is clear from the quotient group relations 
that $\mathcal{P}_q$ is self-dual. 
In fact for $k = 1$ to $9$, this is the self-dual regular $4$-polytope of type $\{6,3q,6\}$ with automorphism group of order $72q$ 
listed at \cite{Conder-Regular}. 

Moreover,  it can be shown that if $q$ is odd, then:
\\[+3pt] 
(a) the medial layer graph $\mathcal{G}$ of $\mathcal{P}_q$ has order $6q$, \\
(b) the automorphism group of $\mathcal{G}$ has order order $2^{2q+2}3^{2q}q$, \\
(c) the associated Cayley graph $\mathcal{C}$ has order $36q$, and \\
(d) the automorphism group of $\mathcal{C}$ has order order $2^{4q+3}3^{4q}q$. 
\\[+4pt]
The proofs of these facts are straightforward but time-consuming, and so we omit them. 
Details are given in Chapter 6 of the second author's Masters thesis~\cite{Steinmann-thesis}, in which it is  
shown that in fact $\mathcal{G}$ is isomorphic to the graph $C(3,2q,1)$ and $\mathcal{C}$ is isomorphic to the graph $C(3,4q,2)$,  
in a family of vertex-transitive graphs $C(p,r,s)$ with large vertex-stabilisers studied by Praeger and Xu in \cite{PraegerXu}. 

In particular, the stabiliser in $\Aut(\mathcal{G})$ of a vertex of $\mathcal{G}$ has order $2^{2q+1}3^{2q-1}$, 
while the stabiliser in $\Aut(\mathcal{C})$ of a vertex of $\mathcal{C}$ has order $2^{4q+1}3^{4q-2}$, 
which is larger than $2^{2q+1}3^{2q-1}$ for all $q$, and very much larger when $q$ is large. 

\smallskip
We also found a similar family of nonorientably-regular polytopes of type $\{4,q,4\}$ with $q = 6t$, for which 
 the vertex-stabilisers for the Cayley graph ${\mathcal C}$ are larger than those for the medial layer graph ${\mathcal G}$. 
For each of these, the automorphism group of the polytope is a smooth quotient of the finitely-presented group  
obtained from the $[4,6t,4]$ Coxeter group with the following extra relations:
\\[+6pt]
${}$\qquad\qquad $ [ x_1, x_2 x_3 x_2 x_3 x_2] = [ x_4, x_3 x_2 x_3 x_2 x_3] = ( x_2 x_1 x_2 x_3 x_4 x_3)^2 $
\\[+4pt] 
${}$\qquad\qquad\quad $  = x_1 x_2 x_3( x_1 x_2)^2 x_4 x_3 x_2( x_3 x_4)^2 = ( x_1 x_2 x_3)^3( x_2 x_3)^{3t-3} = 1$. 
\\[+6pt]
Here in each case the graphs ${\mathcal G}$ and ${\mathcal C}$ are isomorphic to the graphs $C(2,q,2)$ and $C(2,2q,4)$ 
from the same family of graphs $C(p,r,s)$ studied in \cite{PraegerXu}.

\section{Final remarks}
\label{sec:final remarks}

We complete this paper by mentioning a few open questions that naturally arise. 

First, for which values of $q$ do there exist improperly self-dual chiral polytopes of type $\{3,q,3\}$?  
We found examples with $q = 6$ and $q = 18$, but we see no reason why further examples could not exist.  
Similarly, it would be interesting to find examples of improperly self-dual chiral polytopes of type $\{3,q,3\}$ 
that admit polarities. This is guaranteed when $q$ is odd, but could they exist also when $q$ is even?
We know of no such examples. 

The referee suggested also looking for  improperly self-dual chiral polytopes of type $\{p,q,p\}$ for prime $p > 3$. 

Next, if $\mathcal{P}$ is a regular polytope of type $\{3,q,3\}$ for which the medial layer graph $\mathcal{G}$ 
is 3-arc-regular, then the associated Cayley graph can be either 2-arc-regular or 3-arc-regular. 
Are there necessary and sufficient conditions on $\mathcal{P}$ for each of these two possibilities?

Finally, a natural question to ask is whether Theorem \ref{thm: s vs t for cubic} holds also for regular and/or chiral polytopes 
of type $\{p,q,p\}$ for larger values of $p$.  Again, we know no such examples.

\newpage    

\medskip\bigskip\noindent 
{\Large\bf Acknowledgements}

\medskip\smallskip
\noindent 
The authors also acknowledge the helpful use of {\sc Magma}~\cite{Magma} in conducting experiments, testing 
conjectures and verifying many of the details of their construction and proof of the new discoveries presented in this paper.  
Also the first author is grateful to New Zealand's Marsden Fund for its support (via grant UOA 2030).


\end{document}